\documentclass[a4paper]{iopart}
\usepackage[english]{babel}
\usepackage[utf8]{inputenc}
\usepackage[T1]{fontenc}
\usepackage{url, hyperref}
\usepackage[svgnames]{xcolor}
\expandafter\let\csname equation*\endcsname\relax
\expandafter\let\csname endequation*\endcsname\relax
\usepackage{amsmath}
\usepackage{amssymb}
\usepackage{amsthm}
\usepackage{bbm}
\usepackage{caption}
\usepackage{subcaption}
\usepackage{float}
\usepackage{tikz}
\usepackage{bm}

\definecolor{hrefcolor}{rgb}{0.0,0.5,0.8}
\definecolor{hlgreen}{rgb}{0,0.7,0}

\newcommand{\field}[1]{\mathbb{#1}}
\newcommand{\R}{\field{R}}
\newcommand{\iprod}[2]{\langle #1,#2\rangle}

\newcommand{\defeq}{:=}
\DeclareMathOperator*{\argmin}{arg\,min}
\newcommand{\norm}[1]{\|#1\|}

\newcommand{\grad}[1]{\nabla #1}

\newcommand{\freevar}{\,\boldsymbol\cdot\,}

\newtheorem{proposition}{Proposition}

\theoremstyle{definition}

\newtheorem*{assumption*}{Assumption}
\newtheorem{remark}{Remark}
\newtheorem*{remark*}{Remark}
\newtheorem*{definition*}{Definition}

\numberwithin{equation}{section}
\numberwithin{lemma}{section}
\numberwithin{theorem}{section}
\numberwithin{proposition}{section}
\numberwithin{definition}{section}
\numberwithin{remark}{section}
\numberwithin{example}{section}
\numberwithin{assumption}{section}
\numberwithin{algorithm}{section}
\numberwithin{corollary}{section}

\newlength{\w}

\makeatletter
\newcommand*{\centerfloat}{%
  \parindent \z@
  \leftskip \z@ \@plus 1fil \@minus \textwidth
  \rightskip\leftskip
  \parfillskip \z@skip}
\makeatother

\makeatother
\hypersetup{
    pdftitle = {\@title}
    pdfauthor = {Martin Benning, Carola-Bibiane Schönlieb, Tuomo Valkonen, Verner Vlacic}
}
\makeatletter

\begin{document}

\title[Explorations on anisotropic regularisation of dynamic inverse problems]{Explorations on anisotropic regularisation of dynamic inverse problems by bilevel optimisation}

\author{
    Martin Benning$^1$,
    Carola-Bibiane Schönlieb$^1$,
    Tuomo Valkonen$^1$,
    Verner Vla\v{c}i\'{c}$^1$
    }

\address{$^1$ Department of Applied Mathematics and Theoretical Physics, University of Cambridge, United Kingdom}

\begin{abstract}
 We explore anisotropic regularisation methods in the spirit of \cite{holler2013infimal}. Based on ground truth data, we propose a bilevel optimisation strategy to compute the optimal regularisation parameters of such a model for the application of video denoising. The optimisation poses a challenge in itself, as the dependency on one of the regularisation parameters is non-linear such that the standard existence and convergence theory does not apply. Moreover, we analyse numerical results of the proposed parameter learning strategy based on three exemplary video sequences and discuss the impact of these results on the actual modelling of dynamic inverse problems.
\end{abstract}


\maketitle

\section{Introduction}


%
%
%

In this paper, we employ bilevel optimisation to explore different choices of spatial- and temporal regularisation in a variational image reconstruction model. 

Variational regularisation methods are extremely popular and versatile tools when it comes to computing approximate solutions to ill-posed inverse problems. Given the assumption of normal distributed noise, they usually have the form of a generalised Tikhonov-type regularisation, i.e.
\begin{align}\label{eq:varmod}
u_{\bm{\alpha}} \in R(\bm{\alpha}) := \argmin_{u \in L^2(\Omega) \cap \mathcal{U}} \left\{ \frac{1}{2} \| K u - g \|_2^2 + G(u; \bm{\alpha}) \right\} \, \text{,}
\end{align}
where $K \in \mathcal{L}(L^2(\Omega), L^2(\Sigma))$ is a bounded, linear operator mapping from the Hilbert space $L^2(\Omega)$ to the Hilbert space $L^(\Sigma)$, with $\Omega$ and $\Sigma$ being bounded and connected domains. The function $g \in L^2(\Sigma)$ represents the given measurement data, and the norm $\| \cdot \|_2$ simply denotes the $L^2(\Sigma)$-norm. Given a signal $u \in \mathcal{U}$ and regularisation parameters $\bm{\alpha} \in \mathcal{V}$, the functional $G:\mathcal{U} \times \mathcal{V} \rightarrow \mathbb{R}$ represents the regulariser, for Banach spaces $\mathcal{U}$ and $\mathcal{V}$.

Particularly in imaging and image processing applications, proper, lower semi-continuous (l.s.c), convex and non-smooth regularisers have attracted a lot of attention over the last two decades. Various types of total variation regularisation \cite{rudin1992nonlinear} and $\ell^1$ regularisation of unitary transformed signals \cite{donoho2006compressed} have been proposed, in order to exploit sparsity of a signal with respect to a given representation. 

Despite allowing for significant improvements in terms of reconstruction quality, non-smooth regularisation methods suffer from introducing systematic modelling artefacts like any other regularisation method; in case of total variation regularisation for instance, the regularisation method is well-known to introduce piecewise constant approximations of noisy, non-constant regions, which is known as the stair-casing effect (cf. \cite[Section 4.2]{burger2013guide}).

In order to compensate for modelling artefacts, the concept of infimal convolutions can be used for combining the advantages of different regularisers into one. The infimal convolution of two proper, l.s.c. and convex functionals $J_1$ and $J_2$ is defined as
\begin{align}
(J_1 \square J_2)(u) := \inf_{u = v + w} J_1(v) + J_2(w) \, \text{.}\tag{IC}
\end{align}
In \cite{chambolle1997image, benning2013higher, benning2011singular, muller2013advanced} it has been shown that an infimal convolution of the total variation and a higher-order total variation is beneficial for fighting the stair-casing phenomenon, and that infimal convolutions in general are useful in order to reconstruct functions that are additive compositions of functions which can individually be recovered by different regularisers.

Recently, infimal convolution has been considered as a suitable model for handling dynamic inverse problems \cite{holler2013infimal}. Holler and Kunisch have proposed to use infimal convolution in order to combine regularisation functionals that are suitable for either spatial or temporal regularisation, in order to create an appropriate spatio-temporal regulariser. In particular, their model involves a regulariser 
$G$ in \eqref{eq:varmod} which constitutes an infimal-convolution of total variation functionals of weighted spatial and temporal derivatives. This will be specified in Section \ref{sec:dynreg}. The use of an infimal convolution between dominantly spatial and dominantly temporal regularisation terms not only allows the reconstruction of a regularised dynamic image sequence, but also the decomposition of the latter into a sequence of images encoding dominantly temporal information and another sequence encoding dominantly spatial information. 

In what follows we want to learn optimal decompositions in this regularisation in the context of video de-noising by an appropriate bilevel optimisation approach. In the context of \eqref{eq:varmod} the associated bilevel learning problem we will discuss looks for an optimal parameter vector $\bm{\alpha}$ that solves for some convex, proper, weak* lower semicontinuous cost functional $F: X \to \mathbb R$ the problem
\begin{equation}
    \label{eq:learn}
    \tag{$\mathrm{P}$}
    \min_{\alpha \in \mathcal{P}} F(u_{\bm{\alpha}})
\end{equation}
subject to $u_{\bm\alpha}$ being a solution of \eqref{eq:varmod}.

In the context of computer vision and image processing bilevel optimisation is considered as a supervised learning method that optimally adapts itself to a given dataset of measurements and desirable solutions. In \cite{roth2005fields,tappen2007utilizing,domke2012generic,chen2014}, for instance the authors consider bilevel optimization for finite dimensional Markov random field models. In inverse problems the optimal inversion and experimental acquisition setup is discussed in the context of optimal model design in works by Haber, Horesh and Tenorio \cite{haber2003learning,haber2010numerical}, as well as Ghattas et al. \cite{bui2008model,biegler2011large}.  Recently parameter learning in the context of functional variational regularisation models \eqref{eq:varmod} also entered the image processing community with works by the authors \cite{de2013image,calatronidynamic,reyes2015a,reyes2015b,lucainfimal,chungdelosreyes}, Kunisch, Pock and co-workers \cite{kunisch2013bilevel,Chen2012,ochs_ssvm2015}, Chung et al. \cite{chung2014optimal}, Hinterm\"uller et al. \cite{hintermuller2014bilevel} and others \cite{baus2014fully,schmidt2014shrinkage,FNSW15}. Interesting recent works also include bilevel learning approaches for image segmentation \cite{Ranftl_GCPR2014}, learning and optimal setup of support vector machines \cite{Klatzer2015} and learning discrete reaction-diffusion filters \cite{chen_cvpr2015}. 

What we show is closest in flavour to recent applications of bilevel optimisation to supervised learning of optimal parameters in a total variation type regularisation model \cite{delosreyes2014learning,kunisch2013bilevel,tuomov-interior,tuomov-tgvlearn}. The main difference in the theory and computational realisation to these works and the model discussed in this paper is due to the nonlinear dependence of the lower level problem on the parameter vector $\bm\alpha$ that the model is optimised for. To our knowledge, this is the first publication that deals with learning of non-linear regularisation parameters in the context of regularisation of inverse problems.


The paper is organised as follows. First, we are going to recall the concept of spatio-temporal regularisation via infimal convolutions of regularisers. Then, we are going to present the bilevel optimisation framework for learning the regularisation parameters of the infimal convolution regularisers. Subsequently, we are going to address the numerical aspects of the parameter learning strategy. We then conclude with numerical examples and their discussion. We particularly want to address the question of how realistic the assumption of the coupling of the spatial and temporal regularisation is, given three different types of images sequences.

\section{Regularisation of dynamic inverse problems}\label{sec:dynreg}
Let $g=(g_1,\ldots, g_m)\in L^1(\Omega; \R^m)$. We denote by $\norm{\freevar}_{2,1}$ the $L^1(\Omega)$-norm of the two-norm of vector-valued functions, namely
\[
    \norm{g}_{2,1} \defeq \int_\Omega \norm{g(x)}_2 \, dx,
\]
where $\norm{g(x)}_2 = \sqrt{g_1(x)^2+\ldots + g_m(x)^2}$. Based on \cite{holler2013infimal}, we introduce the anisotropic derivative $\nabla_\kappa$ and its negative adjoint $\mathrm{div}_\kappa$ defined for a scalar $\kappa \in (0, 1)$ as
\begin{equation}
\begin{aligned}
\nabla_\kappa & =\left(\kappa \frac{\partial}{\partial x},\kappa \frac{\partial}{\partial y}, (1-\kappa) \frac{\partial}{\partial t} \right), \quad \text{and} \\
\mathrm{div}_\kappa & =\kappa \frac{\partial}{\partial x}+\kappa \frac{\partial}{\partial y} + (1-\kappa) \frac{\partial}{\partial t},
\end{aligned}\label{eq:kappagraddiv}
\end{equation}

\noindent With these, and for $u \in W^{1, 1}(\Omega)$ we define the following dynamic regularisation functionals:
\begin{table}[ht]
\begin{align}
 G\left(u;\alpha_1,\alpha_2,\kappa\right) &= \inf_{u=v+w} \alpha_1\|\nabla_\kappa v\|_{2,1}+\alpha_2\|\nabla_{1-\kappa}w\|_{2,1}\tag{IC-TVTV}\label{eq:ictvtv}\\
G\left(u;\alpha_1,\alpha_2,\kappa\right) &= \inf_{u=v+w} \frac{\alpha_1}{2}\|\nabla_\kappa v\|_2^2+\alpha_2\|\nabla_{1-\kappa}w\|_{2,1}\tag{IC-$L^2$TV}\label{eq:icl2tv}\\
G\left(u;\alpha_1,\alpha_2\right) &= \alpha_1\|\nabla u\|_{2,1}+\alpha_2\|\frac{\partial}{\partial t}u\|_{1} \tag{Rigid TVTV}\label{eq:rigidtvtv}\\
G\left(u;\alpha_1,\alpha_2\right) &= \frac{\alpha_1}{2}\|\nabla u\|_2^2+\alpha_2\|\frac{\partial}{\partial t}u\|_{1} \tag{Rigid $L^2$TV}\label{eq:rigidl2tv}
\end{align}
\caption{The different dynamic regularisers used throughout this paper.}
\label{tab:dynreg}
\end{table}
Here \eqref{eq:ictvtv} is a direct adaptation of the ICTV functional proposed in \cite{holler2013infimal}. Regulariser \eqref{eq:icl2tv} is a modification that allows for different regularisation models; in case we have picked two rather complementary regularisation functionals, with the L$^2$ norm of the gradient complementing the total variation regularisation. Regularisers \eqref{eq:rigidtvtv} and \eqref{eq:rigidl2tv} can almost be seen as limiting cases of \eqref{eq:ictvtv} and \eqref{eq:icl2tv}, respectively. Choosing $\kappa \in \{0, 1\}$ and restricting solutions to $v = w$ converts \eqref{eq:ictvtv} and \eqref{eq:icl2tv} into \eqref{eq:rigidtvtv} and \eqref{eq:rigidl2tv}.

The basic motivation for these models is a decomposition of a dynamic image sequence into a spatial and a temporal component, such that these components are penalised individually with suitable regularisation functionals. However, in order to allow for space-time correspondence in these sequences, an additional anisotropy parameter $\kappa$ is introduced that ensures neither one of the components to be penalised by the spatial or the temporal regulariser alone. Speaking of spatial and temporal components, a spatial component is regularised in space only, whereas the temporal component is regularised in only in time. In \eqref{eq:kappagraddiv} this corresponds to the extreme cases $\kappa = 0$ (only temporal regularisation) and $\kappa = 1$ (only spatial regularisation). The restriction to $\kappa \in (0, 1)$, which further ensures spatial and temporal regularity in both components, also guarantees well-posedness of \eqref{eq:varmod}.

A major challenge for successful regularisation of dynamic image sequences via any of the regularisers listed in Table \ref{tab:dynreg} is the 'optimal' choice of parameters $\alpha_1$, $\alpha_2$ and $\kappa$. On the one hand, this is due to the number of parameters itself, making it difficult to employ heuristic parameter choice rules that may succeed in case of single parameters. On the other hand, the dependency of the model on the regularisation parameter $\kappa$ is non-linear, making it even harder to optimise for.

Following \cite{delosreyes2014learning,kunisch2013bilevel,tuomov-interior,tuomov-tgvlearn}, we propose a learning framework that allows to learn the regularisation parameters from training data. As mentioned before, the key difference is that the dependency of $\kappa$ is non-linear. 


\section{Optimising the parameters---the theory}

It is not immediately clear, which parameter choices for $\alpha_1$, $\alpha_2$ and $\kappa$ may be optimal for dynamic inverse problems regularised by one of the anisotropic regularisers given in Table \ref{tab:dynreg}. Parameter learning approaches, as discussed in the introduction, provide a means towards studying optimal parameter choices with respect to a known ground truth. We concentrate in particular on the bilevel optimisation framework, first presented in \cite{delosreyes2014learning,kunisch2013bilevel} and further studied in the context of multi-parameter regularisers in \cite{tuomov-interior,tuomov-tgvlearn}. The general form therein is given as
\begin{equation}
    \min_{\bm{\alpha} \in \mathcal{P}}~ \frac{1}{2}\norm{R(\bm{\alpha})-g}_2^2 \quad\text{s.t.}\quad R(\bm{\alpha})=\argmin_{u \in L^1(\Omega)}~ \frac{1}{2}\|f-Ku\|_2^2+ G\left(u;\bm{\alpha}\right).
    \label{eq:bilevopt}
\end{equation}
Here $\bm{\alpha}$ denotes the vector of parameters we are optimising for, and $\mathcal{P}$ is our space of allowed parameters; generally $\bm{\alpha}=(\alpha_1,\alpha_2, \kappa)$ with $\mathcal{P}= [0, \infty)^2 \times [0, 1]$ for the infimal convolution regularisers, and $\bm{\alpha}=(\alpha_1, \alpha_2)$ with $\mathcal{P}= [0, \infty)^2$  for the rigid regularisers. Here, $\norm{R(\bm{\alpha})-g}_2$ is the cost functional $F(R(\bm\alpha))$ measuring the distance of the denoised solution $R(\bm{\alpha})$ from the ground truth original $g$, and $G$ is a regulariser for the lower level problem of reconstructing $u$ from noisy $f$. 


\subsection{Existence of solutions}
A general existence and convergence theory for solutions $\bm{\alpha}$ to \eqref{eq:bilevopt} is presented in \cite{tuomov-interior} when $G$ is linear in $\bm{\alpha}$ and defined over the space of bounded variation functions, i.e. $\| \nabla u \|_{2, 1}$ in Table \ref{tab:dynreg} is replaced by 
\begin{align*}
\sup_{\substack{\varphi \in C_0^\infty(\Omega; \R^m)\\ \| \varphi \|_{2, \infty} \leq 1}} \int_\Omega u(x) \, (\mathrm{div}_\kappa \varphi)(x) \, dx,
\end{align*}
where the notation $\| \cdot \|_{2, \infty}$ is analogous to $\| \cdot \|_{2, 1}$, but with the $L^\infty(\Omega)$-norm replacing the $L^1(\Omega)$-norm. Our regularisers are not linear in $\kappa$ however, so an extended theory would be needed. With an additional elliptic regularisation, however, we can still show existence of solutions easily.

\begin{proposition}
    \label{prop:existence}
    Consider the problem 
    \begin{equation}
        \min_{\bm{\alpha} \in \mathcal{P}}~ \frac{1}{2}\norm{R(\bm{\alpha})-g}_2^2 \quad\text{s.t.}\quad R(\bm{\alpha})=\argmin_{u \in L^1(\Omega)}~ \frac{1}{2}\|f-Ku\|_2^2+ G\left(u;\bm{\alpha}\right) + \frac{\epsilon}{2}\norm{\grad u}^2,
        \label{eq:bilevopt-smooth}
    \end{equation}
    where $G$ is one of the regularisers from Table \ref{tab:dynreg}, $\mathcal{P}$ the corresponding admissible set of parameters, and $\epsilon>0$.
    Suppose constant functions are not in the null space of $K$.
    Then there exists an optimal solution $\bm{\alpha} \in \mathcal{P}$ to \eqref{eq:bilevopt-smooth}.
\end{proposition}

\begin{proof}
    Note that all regularisers presented in Table \ref{tab:dynreg} are lower semi-continuous with respect to the mutual convergence of $\bm{\alpha}$ in $\R^3$ and of $u$ in $L^1$. They are moreover continuous with respect to $\bm{\alpha}$ for fixed $u$.

\noindent Let us first consider the inner problem
    \begin{equation}
        \label{eq:inner-smooth}
        \argmin_{u \in L^1(\Omega)}~ J(u; \bm{\alpha}) \defeq \frac{1}{2}\|f-Ku\|_2^2+ G\left(u;\bm{\alpha}\right) + \frac{\epsilon}{2}\norm{\grad u}^2
    \end{equation}
first. The term $\frac{\epsilon}{2}\norm{\grad u}^2$ and the fact that constant functions are not in the kernel of $K$ guarantees weak convergence in $H^1(\Omega)$ of a (subsequence of a) minimising sequence $\{u^k\}$ regardless of the choice of $\bm{\alpha}$ in the inner problem.
    This implies the convergence in $L^1$, and consequently by the standard argument of calculus of variations, an existence of a solution $R(\bm{\alpha})$ to the inner problem.

    For a minimising sequence $\{\bm{\alpha}^k\}$ of the whole problem \eqref{eq:bilevopt-smooth} we therefore get weak convergence in $L^2$ of $u^k \defeq R(\bm{\alpha}^k)$ to some $\hat u$.
    But since $\norm{u^k-g}_2^2 \to \norm{u-g}_2^2$ for such a minimising sequence, we have so-called strict convergence of $u^k-g$ to $u-g$. In $L^2$ this implies strong convergence of first $u^k-g$ to $u-g$, and then of $u^k$ to $\hat u$. 
    Suppose also $\bm{\alpha}^k \to \hat{\bm{\alpha}}$.
    We then compute
    \[
        \begin{split}
        J(u; \hat{\bm\alpha})
        &
        =
        \liminf_{k \to \infty} \left( J(u; {\bm\alpha}^k) + G(u; \hat{\bm\alpha}) - G(u; {\bm\alpha}^k) \right)
        \\
        &
        \ge
     \liminf_{k \to \infty} J(u; {\bm\alpha}^k) + \liminf_{k \to \infty} \left( G(u; \hat{\bm\alpha}) - G(u; {\bm\alpha}^k) \right).
        \end{split}
    \]
    Using the continuity of $G$ with respect to $\bm{\alpha}$, we therefore obtain
    \[
        J(u; \hat{\bm\alpha})
        =
        \liminf_{k \to \infty} J(u; {\bm\alpha}^k)
        \ge
        \liminf_{k \to \infty} J(u^k; {\bm\alpha}^k)
        \ge
        J(\hat u; \hat{\bm\alpha}).
    \]
    This shows that $\hat u=R(\hat{\bm\alpha})$, and hence that $\hat{\bm\alpha}$ solves \eqref{eq:bilevopt-smooth}.
\end{proof}

\begin{remark}
    \label{rem:epsilon-alpha}
    Note that we only used $\epsilon$ to show existence of solutions to the inner problem. In particular, we do not require $\epsilon>0$ to be constant, but we can allow $\epsilon=\epsilon(\bm{\alpha}$) to vary continuously with respect to $\bm{\alpha}$.
\end{remark}

\subsection{Derivative of the solution map}

In order to use standard optimisation methods, such as BFGS, to minimise \eqref{eq:bilevopt}
we need to calculate a gradient of the solution map $R$; equivalently, following the PDE approach of \cite{delosreyes2014learning}, we need to solve a so-called adjoint equation. 
A solution is given by the classical implicit function theorem \cite[Theorem 4.E]{zeidler2012applied}, and in particular the version in \cite[Corollary 4.34]{mordukhovich2006variational} with relaxed assumptions.
Let us suppose $G$ is differentiable, such that $R(\bm{\alpha})$ is given as the solution $u=u_{\bm{\alpha}}$ to
\[
    0 = S(u, \bm{\alpha}) \defeq K^*(Ku-f) + \grad_u G(u; \bm{\alpha}).
\]
Then, if $S$ is strictly differentiable, and $\grad_u S(u, \bm{\alpha})$ is invertible, we have
\begin{equation}
    \label{eq:solmapderiv}
    \grad R(\bm{\alpha})
    =
    [\grad_u S(u, \bm{\alpha})]^{-1} \grad_{\bm{\alpha}} S(u, \bm{\alpha}).
\end{equation}
The strict differentiability of $S$ may be achieved by a ``second-degree'' Huber regularisation \cite{de2011optimal}. The invertibility of $\grad_u S(u, \bm{\alpha})$ can be guaranteed by an additional elliptic regularisation term $\frac{\epsilon}{2}\norm{\grad u}_2^2$.
Both extra regularisations are the same as already employed in \cite{delosreyes2014learning}, where an alternative adjoint equation route was taken to avoid direct construction of $\grad R$. We now take a closer look at the derivatives of the solution maps for the regularisers \eqref{eq:ictvtv} and \eqref{eq:icl2tv}. Furthermore, as we are going to restrict ourselves to the regularisation of dynamic video sequences, we consider $K$ to be the identity operator mapping from $L^2(\Omega)$ to $L^2(\Omega)$ throughout the remainder of this paper.

\subsection{Derivative of the IC TVTV solution map}
\label{sec:ictvtv-theory}

To use the formula \eqref{eq:solmapderiv}, we need the derivative of the map
\begin{equation}
    \notag
    R(\alpha_1,\alpha_2,\kappa)=P_1 \hat R(\alpha_1, \alpha_2, \kappa)
\end{equation}
for $P_1(u, w) \defeq u$, and
\[
    \hat R(\alpha_1, \alpha_2, \kappa)
    \defeq
    \argmin_{(u,w)}~ \frac{1}{2}\|u-g\|_2^2+\alpha_1 \|\nabla_\kappa(u-w)\|_{2,1}+\alpha_2\|\nabla_{1-\kappa}(w)\|_{2,1}.
\]
Clearly
\begin{equation}
    \label{eq:r-hat-r-grad}
    \grad R(\alpha_1,\alpha_2,\kappa)=P_1 \grad \hat R(\alpha_1,\alpha_2,\kappa)
\end{equation}
if the latter derivative exists.
For $\hat R(\alpha_1, \alpha_2, \kappa)$ to have a unique minimiser and to be differentiable, we further replace $\hat R(\alpha_1, \alpha_2, \kappa)$ with
\begin{equation}
    \label{eq:tvtvsolmap-reg}
    \hat R_{\gamma,\epsilon}(\alpha_1,\alpha_2,\kappa)=\argmin_{(u, w)}~ J_{\gamma,\epsilon}(u,w,\alpha_1,\alpha_2,\kappa)
\end{equation}
for
\[
    J_{\gamma,\epsilon}(u,w,\alpha_1,\alpha_2,\kappa)
    \defeq
    \frac{1}{2}\|u-g\|_2^2+\alpha_1 \Theta(\nabla_\kappa(u-w))
    +\alpha_2\Theta(\nabla_{1-\kappa}(w))
    +\frac{1}{2} \left(\int_\Omega w \,dx\right)^2.
\]
Here we use the notation
\begin{equation}
    \label{eq:theta}
    \Theta(v) :=\|v\|_\gamma+\frac{\epsilon}{2}\|v\|_2^2
\end{equation}
for the sum of the Huber-regularised $1$-norm $\| v \|_\gamma = \int_\Omega H_\gamma(\| v(x) \|_2) \, dx$ with
\begin{align*}
H_\gamma(r) = \begin{cases} | r | - \frac{\gamma}{2} & | r | \geq \gamma \\ \frac{1}{2\gamma} | r |^2 & | r | < \gamma \end{cases}
\end{align*}
and parameter $\gamma>0$, and additional Hilbert space regularisation with parameter $\epsilon>0$.
The last term of \eqref{eq:tvtvsolmap-reg} eliminates the translational invariance in $\R$ with respect to $w$, thus forcing unique solutions as needed for the application of \eqref{eq:solmapderiv}.

Note that existence of a solution $(u, w)$ to \eqref{eq:tvtvsolmap-reg} is guaranteed by simple application of Proposition \ref{prop:existence} and Remark \ref{rem:epsilon-alpha} provided $\kappa \not\in \{0,1\}$. In that case $\Theta(\grad_\kappa u) \ge \epsilon'\norm{\grad u}_2^2$ for some $\epsilon'=\epsilon'(\kappa,\epsilon)$.

For the following propositions we define the shorthand notations
\begin{align}
    \notag
    \Psi^1_\kappa(u) & \defeq [\grad \Theta](\grad_\kappa(u-w)),
    &
    \Psi^2_\kappa(u) & \defeq [\grad^2 \Theta](\grad_\kappa(u-w))    
    \\
    \intertext{and}
    \notag
    Q &  \defeq \left(\begin{smallmatrix}I & -I \\ 0 & I\end{smallmatrix}\right),
\end{align}
where we denote by $I$ the identity operator. We also model by $c \defeq \chi_\Omega$ the term $\int_\Omega w \, dx = \iprod{c}{w}$  in \eqref{eq:tvtvsolmap-reg}.

Now we can derive the derivatives of the solution map.

\begin{proposition}[Derivative of the IC TVTV solution map]
    Let $\hat R_{\gamma,\epsilon}$ be defined by \eqref{eq:tvtvsolmap-reg} for some $\gamma,\epsilon>0$, and
    \[
            R(\alpha_1,\alpha_2,\kappa) \defeq P_1 \hat R_{\gamma,\epsilon}(\alpha_1, \alpha_2, \kappa).
    \]
    Then
    \begin{subequations}
    \begin{equation}
        \grad R(\alpha_1,\alpha_2,\kappa)
        = -P_1 \left(\frac{\partial S}{\partial(u,w)}\right)^{-1}\frac{\partial S}{\partial(\alpha_1,\alpha_2,\kappa)},
    \end{equation}
    where $S=\grad_{(u,w)} J_{\gamma,\epsilon}$ satisfies
    \label{eq:tvtvderivative}
    {\footnotesize
    \begin{align}
    \frac{\partial S}{\partial(u,w)}& \defeq Q^*\begin{pmatrix}I- \alpha_1 \mathrm{div}_\kappa \Psi^2_\kappa (u-w) \nabla_\kappa & I\\ I & I-\alpha_2 \mathrm{div}_{1-\kappa} \Psi^2_{1-\kappa}(w) \nabla_{1-\kappa}+c \otimes c \end{pmatrix}Q, \label{eq:tvtvderivativematrix}
    \\
        \frac{\partial S}{\partial(\alpha_1,\alpha_2,\kappa)}& \defeq Q^*
        \begin{pmatrix}-\mathrm{div}_\kappa \Psi^1_\kappa(u-w) & 0 & T_1 
        \\
        0 & -\mathrm{div}_{1-\kappa} \Psi^1_{1-\kappa}(w) & T_2
    \end{pmatrix},
    \label{eq:tvtvderivativevector}
    \end{align}
    }
    where
    {\footnotesize
    \begin{align}
        T_1 & \defeq -\alpha_1\left(\mathrm{div}-\frac{\partial}{\partial t^*}\right)\Psi^1_\kappa(u-w) - \alpha_1 \mathrm{div}_\kappa \Psi^2_\kappa(u-w) \left(\nabla-\frac{\partial}{\partial t}\right)(u-w), 
        \quad\text{and}\\
        T_2 & \defeq
      -\alpha_2\left(\frac{\partial}{\partial t^*}-\mathrm{div}\right)\Psi^1_{1-\kappa}(w) - \alpha_2 \mathrm{div}_{1-\kappa} \Psi^2_{1-\kappa}(w) \left(\frac{\partial}{\partial t}-\nabla\right)w.
    \end{align}
    }
    \end{subequations}
\end{proposition}

The problem of finding the derivative of the solution map thus reduces to the problem of solving three linear systems with the operator \eqref{eq:tvtvderivativematrix} for the right-hand side vectors the columns of \eqref{eq:tvtvderivativevector}.

\begin{proof}
The optimality conditions for the solution $u=R_{\gamma,\epsilon}(\alpha_1,\alpha_2,\kappa)$ are obtained from its Euler-Lagrange equation $S(u, w, \alpha_1,\alpha_2,\kappa)=0$, which we split into
\begin{subequations}
\begin{align}
S_1& \defeq u-f-\alpha_1 \mathrm{div}_\kappa \Psi^1_\kappa(u-w)=0, \quad\text{and}\\
S_2& \defeq \alpha_1 \mathrm{div}_\kappa \Psi^1_\kappa(u-w)-\alpha_2 \mathrm{div}_{1-\kappa} \Psi^1_{1-\kappa}(w) +\int_\Omega w \,dx=0\label{eq:tvtvS2_definition}
\end{align}
\end{subequations}
Now $(S_1,S_2)=0$ defines $(u,w)$ as an implicit function of $(\alpha_1,\alpha_2,\kappa)$. 
By \eqref{eq:solmapderiv}, the Jacobian of the function $\hat R_{\gamma,\epsilon}: (\alpha_1,\alpha_2,\kappa) \mapsto (u,w)$ is given by
\begin{equation}
\frac{\partial(u,w)}{\partial(\alpha_1,\alpha_2,\kappa)}=-\left(\frac{\partial S}{\partial(u,w)}\right)^{-1}\frac{\partial S}{\partial(\alpha_1,\alpha_2,\kappa)}\label{eq:ictvtvsolmapderiv}
\end{equation}
Calculating all the partial derivatives and using \eqref{eq:r-hat-r-grad}, we obtain \eqref{eq:tvtvderivative}.
\end{proof}


\subsection{Derivative of the IC L\textsuperscript{2}TV solution map}
\label{eq:icl2tv-theory}

Again we have a problem of the form
\begin{equation}
    \notag
    R(\alpha_1,\alpha_2,\kappa)=P_1 \hat R(\alpha_1, \alpha_2, \kappa)
\end{equation}
for $P_1(u, w) \defeq u$, and
\begin{equation}
    \label{eq:l2tvsolmap}
    \hat R(\alpha_1,\alpha_2,\kappa)=\argmin_{(u,w)}~ 
   \frac{1}{2}\|u-g\|_2^2+\frac{\alpha_1}{2} \|\nabla_\kappa(u-w)\|_{2,1}^2+\alpha_2\|\nabla_{1-\kappa}(w)\|_{2,1}.
\end{equation}
To employ the formula \eqref{eq:solmapderiv}, we need to regularise $\hat R(\alpha_1,\alpha_2,\kappa)$ by replacing it with
\begin{equation}
    \label{eq:l2tvsolmap-reg}
    \hat R_{\gamma,\epsilon}(\alpha_1,\alpha_2,\kappa)=\argmin_{u,w}~  J_{\gamma,\epsilon}(u,w,\alpha_1,\alpha_2,\kappa)
\end{equation}
for
\[
    J_{\gamma,\epsilon}(u,w,\alpha_1,\alpha_2,\kappa)
    \defeq
    \frac{1}{2}\|u-g\|_2^2+\frac{\alpha_1}{2} \|\nabla_\kappa(u-w)\|_2^2
    +\alpha_2\Theta(\nabla_{1-\kappa}(w))+\frac{1}{2} \left(\int_\Omega w \,dx\right)^2
\]
Similar to Section \ref{sec:ictvtv-theory}, we obtain the following result.

\begin{proposition}[Derivative of the IC L$^2$TV solution map]
    Let $\hat R_{\gamma,\epsilon}$ be defined by \eqref{eq:l2tvsolmap-reg} for some $\gamma,\epsilon>0$, and
    \[
        R(\alpha_1,\alpha_2,\kappa) \defeq P_1 \hat R_{\gamma,\epsilon}(\alpha_1, \alpha_2, \kappa).
    \]
    Then
    \begin{subequations}
    \begin{equation}
        \grad R(\alpha_1,\alpha_2,\kappa)
        = -P_1 \left(\frac{\partial S}{\partial(u,w)}\right)^{-1}\frac{\partial S}{\partial(\alpha_1,\alpha_2,\kappa)},
    \end{equation}
    where $S=\grad_{(u,w)} J_{\gamma,\epsilon}$ satisfies
    \label{eq:l2tvderivative}
    \begin{align}
        \frac{\partial S}{\partial(u,w)}& \defeq Q^* \begin{pmatrix}I- \alpha_1 \mathrm{div}_{\kappa}\nabla_\kappa & I\\ I & I-\alpha_2 \mathrm{div}_{1-\kappa} \Psi^2_{1-\kappa}(w) \nabla_{1-\kappa}+c \times c\end{pmatrix}Q
        \label{eq:l2tvderivativematrix}
        \\
        \frac{\partial S}{\partial(\alpha_1,\alpha_2,\kappa)}& \defeq Q^* \begin{pmatrix}-\mathrm{div}_{\kappa}\nabla_\kappa (u-w) & 0 & T_1
        \\
        0 & -\mathrm{div}_{1-\kappa} \Psi^1_{1-\kappa}(w) & T_2 \end{pmatrix}
        \label{eq:l2tvderivativevector}
    \end{align}
    where
    \begin{align}
        T_2 & \defeq  -\alpha_2\left(\frac{\partial}{\partial t^*}-\mathrm{div}\right)\Psi^1_{1-\kappa}(w) - \alpha_2 \mathrm{div}_{1-\kappa} \Psi^2_{1-\kappa}(w) \left(\frac{\partial}{\partial t}-\nabla\right)w,
        \\
        T_2 & \defeq
        -\alpha_1\left(\mathrm{div}-\frac{\partial}{\partial t^*}\right)\nabla_\kappa (u-w) - \alpha_1 \mathrm{div}_{\kappa} \left(\nabla-\frac{\partial}{\partial t}\right)(u-w).
    \end{align}
    \end{subequations}
\end{proposition}
Here $Q$ and $c$ are as in Section \ref{sec:ictvtv-theory}. 

\begin{proof}
As in Section \ref{sec:ictvtv-theory}, we find the optimal conditions for the minimization problem in \eqref{eq:l2tvsolmap-reg} are
\begin{subequations}
\begin{align}
S_1& \defeq u-f-\alpha_1 \mathrm{div}_{\kappa} \nabla_\kappa (u-w)=0\\
S_2& \defeq \alpha_1 \mathrm{div}_{\kappa} \nabla_\kappa (u-w)-\alpha_2 \mathrm{div}_{1-\kappa} \Psi^1_{1-\kappa}(w) +\int_\Omega w \,dx=0\label{eq:l2tvS2_definition}
\end{align}
\end{subequations}
The Jacobian of the function $R_{\gamma,\epsilon}: (\alpha_1,\alpha_2,\kappa)\to u$ is again given by
\begin{equation}
\frac{\partial u}{\partial(\alpha_1,\alpha_2,\kappa)}=-P_1 \left(\frac{\partial S}{\partial(u,w)}\right)^{-1}\frac{\partial S}{\partial(\alpha_1,\alpha_2,\kappa)}.\label{eq:icl2tvsolmapderiv}
\end{equation}
We thus calculate the partial derivatives with respect to all the variables to obtain \eqref{eq:l2tvderivative}.
\end{proof}

\section{Computational realisation}

The upper level problem of the bilevel optimisation \eqref{eq:bilevopt} is solved numerically via the BFGS algorithm with backtracking line search and curvature verification, where we update the quasi-Hessian only if it remains positive semi-definite. The following Armijo condition
\begin{equation}
F\left(\bm{\alpha}+\sigma \bm{d}\right)-F\left(\bm{\alpha}\right)\leq \sigma c\, \nabla F \left(\bm{\alpha}\right) \cdot \bm{d}\label{eq:armijocond}
\end{equation}
is used, where $F(\bm{\alpha}) = \frac{1}{2} \| R(\bm{\alpha}) - g \|_2^2$, $\bm{d}$ is the search direction, $\sigma$ the step-length and $c$ a positive constant. The relative residual is used as a stopping criterion, so that the algorithm terminates if
\begin{equation}
\|\bm{\alpha}_i-\bm{\alpha}_{i-1}\|_2<\rho \,\|\bm{\alpha}_i\|_2\label{eq:bfgsstopcrit}
\end{equation}
is satisfied for a fixed, positive parameter $\rho$. As all models have to be differentiable in order to solve the upper level problem, we used the $L^2$ norm $\frac{\epsilon}{2} \|\cdot\|_2^2$ with $\epsilon = 10^{-8}$ and the Huberised $L^1$ norm $\|\cdot\|_\gamma$ with $\gamma=0.01$ as our cost functionals in order to optimise the parameters $(\alpha_1, \alpha_2, \kappa)$ for the regularisers in Table \ref{tab:dynreg}. We have further used MATLAB's inbuilt \texttt{gmres} function with a diagonally compensated incomplete Cholesky preconditioner in order to solve \eqref{eq:ictvtvsolmapderiv} and \eqref{eq:icl2tvsolmapderiv}, respectively.

To compute numerical solutions $u$ (and $w$) of the lower-level problem for fixed $\bm{\alpha}$, the lower level problem of the bilevel optimisation \eqref{eq:bilevopt} is solved via the primal-dual hybrid gradient method (PDHGM) as presented in \cite{chambolle2011first}. In order to apply the PDHGM to the lower level problem, we need to recast it into a saddle-point formulation. In case of \eqref{eq:icl2tv} for instance, this saddle-point formulation reads as
\begin{equation}
\min_{(u,w)}{\max_{(p,q)}{\frac{1}{2}\|u-g\|_2^2+\left\langle A\left(\begin{smallmatrix}u\\w \end{smallmatrix}\right),\left(\begin{smallmatrix}p\\q \end{smallmatrix}\right)\right\rangle-{\textstyle \frac{1}{2 \alpha_1}}\|p\|_2^2-\delta_{\alpha_2 P}(q)}},
\end{equation}
where $P=\left\{p \ | \ \| p \|_{2, \infty} \leq 1 \right\}$ is the unit ball with respect to the supremum norm and $A$ is a linear operator given by
\begin{equation}
A=\begin{pmatrix}\nabla_{\kappa} & -\nabla_{\kappa} \\ 0 & \nabla_{1-\kappa}\end{pmatrix} \, \text{.}
\end{equation}
To determine the step sizes in the PDHGM, we need to find a bound on $\|A\|$. Assuming ${\kappa\in [0,1]}$, it is easy to show that $\|A\|^2<24$. The other formulations can be cast to saddle-point formulations in the same fashion.

\noindent In order to discretise $\nabla_\kappa$ and $\nabla_{1 - \kappa}$, we simply use forward finite differences.


Note that the parameters $\epsilon,\gamma$ in \eqref{eq:theta} are numerical regularisation parameters only. As we run a relatively small number of PDHGM iterations to solve each inner problem, the solutions will be numerically inaccurate. This therefore allows us to ignore $\gamma$ and $\epsilon$ in the inner denoising problem, and to solve the original non-smooth problem instead. In the outer problem, we further do not restrict $\kappa \in [0, 1]$, as this is not strictly necessary. For reporting the results in a uniform fashion, we use the identity
\begin{equation}
\|\nabla_\kappa v\|_1=|2\kappa-1|\, \|\nabla_{\frac{\kappa}{2\kappa-1}} v\|_1,
\label{eq:negative-trick}
\end{equation}
which holds for the unsmoothed inner problems. Therefore every triple $(\alpha_1,\alpha_2,\kappa)$ with $\kappa\notin [0,1]$ corresponds to a triple with $\kappa\in (0,1)$.
\section{Numerical results}


For the implementation of the BFGS scheme we use the following numerical setup. The constant $c$ in \eqref{eq:armijocond} is set to $c=10^{-4}$ throughout all experiments, whereas we use $\rho = 10^{-8}$ in \eqref{eq:bfgsstopcrit}. In all cases the parameters $\alpha_1$ and $\alpha_2$ were constrained to lie in $(10^{-5},100)$, and $\kappa$ was constrained to lie in $(-50,50)$. We ran BFGS with 100 different initialisations of the parameters $\bm{\alpha}$ drawn from a 3rd order $\chi^2$-squared distribution with mean rescaled to the values $(0.15, 0.15)$ for TVTV and $(3.9,0.15)$ for $L^2$TV, obtained by previous experimentation. The $\chi^2$-squared distribution is used in order to force the sampled parameters to be positive, but to not impose an upper bound.
We report the afterwards optimised parameter  $\bm{\alpha}$ for which we obtained the best PSNR values for visual and quantitative comparison in the following figures and tables.

%

In order to solve the lower level problem, the PDHGM is run with a fixed number of iterations - 50 iterations in case the accelerated variant is applicable, otherwise 200 iterations, respectively. These were applied to the non-relaxed variants of the regularisers, as this allows us to leave $\kappa$ unconstrained (see \eqref{eq:negative-trick}). Each iteration uses warm initialisation with the optimal solution from the previous iteration. 

We use three different $2 + 1D$ video sequences for our denoising experiments: the sequences 'hand', 'flight' and 'harlem shake'. The sequence 'hand', showing a hand falling onto a table, is of grid-size $54 \times 96$ and consists of 65 frames. It contains steady objects (like the table) and a moving object (the hand) that, at first not seen, moves onto the table where it becomes a steady object for the rest of the scene. The camera is fixed to a specific position.
The sequence 'flight' is filmed from a flying glider. Here the background scenery passes by, while also the camera is at movement. The movie has grid-size $96 \times 54$ and consists of 90 frames. The third sequence 'harlem shake' is taken from \url{https://www.youtube.com/watch?v=-_ZG2xgNAr4}. The original RGB video has grid-size $540 \times 720$ and consists of 711 frames. It has been converted to gray-scale double precision, and down-sampled to grid-size $70 \times 93$ and 73 frames in order to make processing in a reasonable amount of time possible. The video shows a room inside a lodge in which the majority of people are in a rather steady position, whereas one person is dancing (and therefore moving). Approximately half-way through the video the scene changes, and everyone is at movement. As in case of the video 'hand', the camera is in a fixed position.

For all videos the underlying mesh-sizes are considered to be $h = 1$ in each dimension. Further, noisy versions of the video sequences have been created by perturbing the original sequences with Gau\ss ian noise with mean zero and variance $\sigma^2=0.02$ respectively.

\subsection{Discussion of results}
We want to start discussing the results starting with the video 'hands'. Figure \ref{subfig:hand1} shows six frames that are selected  from the original sequence at the times displayed, and its noisy counterparts for variance $\sigma^2 = 0.02$. We clearly see the features described in the previous section, starting with a relatively steady scene and a moving hand emerging half-way through. Figure \ref{subfig:hand1} shows the spatial and temporal components of the TV-TV reconstruction (\eqref{eq:bilevopt-smooth} with \eqref{eq:ictvtv} as regulariser) with optimal parameters that are given in Table \ref{table:hand}, as well as the sum of both components. Note that we define the temporal component as $u$ if $\kappa \le 0.5$, and $u-w$ otherwise. We observe that in particular for the steady parts of the scene a lot of the noise has been removed. Most of the remaining artefacts seem to be present in the moving object, the hand. The temporal component seems to mostly contain the moving parts of the hand, whereas the spatial component contains a rather piecewise constant transition from desk without to desk with hand.
\begin{table}
    \centering
    \caption{``Hand'' test video optimal results. The star $^*$ in the optimal parameter means that the original $\kappa$ for the optimal result was outside the range $[0, 1]$, and the conversion \eqref{eq:negative-trick} has been used to derive the presented values.}
    \label{table:hand}
    \begin{tabular}{r|r|r|r|r}
        Model & $(\alpha_1, \alpha_2, \kappa)$ & Opt.~value & PSNR & SSIM \\
        \hline
        TVTV & (0.162, 0.0844, 0.0466)$^*$ & 104.5 & 32.08 & 0.9271 \\
        L2TV & (9.85, 0.0674, 0.0529)$^*$ & 127.5 & 31.22 & 0.9125 \\
    \end{tabular}
\end{table}
Figure \ref{subfig:hand2} on the other hand shows the result of the numerical reconstruction of \eqref{eq:bilevopt-smooth} with \eqref{eq:icl2tv} as a regulariser, for optimal parameters that are also given in Table \ref{table:hand}. The results are similar to the TV-TV case; however, the spatial component shows a much smoother transition and therefore picks up the hand much earlier than TV-TV does. This also explains the quality-measure results in Table \ref{table:hand}, as those tell us that TV-TV outperforms $L_2$-TV for this sequence both in terms of PSNR and SSIM.

\newlength{\iw}
\newlength{\iwm}
\newlength{\ih}
\newlength{\stripw}
\newlength{\dotd}
\newlength{\dotw}
\newlength{\doth}
\newlength{\up}
\newlength{\down}
\setlength{\dotd}{3ex}
\setlength{\dotw}{1ex}
\setlength{\doth}{0.67\dotw}
\newcount{\numdots}

\newcommand{\strip}[6]{
    \pgfmathsetlength{\stripw}{\numi*\iwm+\iwm-\iw};
    \pgfmathsetlength{\up}{#4+\ih/2+1ex};
    \pgfmathsetlength{\down}{#4-\ih/2-1ex};
    \fill[color=black,above,right] (-1pt, \up) rectangle (\stripw, \down);
    \pgfmathsetcount{\numdots}{floor(\stripw/\dotd)};
    \foreach \i in {0,...,\numdots} {
        \fill[color=white,below,right] (\i*\dotd+3pt, \up-1.5pt) rectangle (\i*\dotd+\dotw+3pt, \up-\doth-1.5pt);
        \fill[color=white,below,right] (\i*\dotd+3pt, \down+1.5pt) rectangle (\i*\dotd+\dotw+3pt, \down+\doth+1.5pt);
    }
    \pgfkeys{/pgf/number format/precision=1}
    \foreach \x/\frame in #3 {
        \pgfmathsetmacro{\xminusone}{\x-1}
        \node[rotate=90] at (-2ex,#4) {#5};
        \node[below,right] at (\xminusone*\iwm, #4) {\includegraphics[width=\iw]{{#1_frame\frame}.png}};
        \def\temp{#6}
        \ifx\temp\empty\else
            \pgfmathsetmacro{\frametime}{\frame/#6}
            \node[below,right,color=orange!70!white] at (\xminusone*\iwm, #4-0.5\ih+2.1ex) {\pgfmathprintnumber[fixed,zerofill,precision=2]{\frametime}s};
        \fi
    }%
}

\def\aspect{0.667}

\newcommand{\setupvideo}[3]{%
    \pgfmathsetmacro{\numi}{#2}
    \pgfmathsetlength{\iw}{\textwidth/\numi}%
    \pgfmathsetlength{\iwm}{\iw+2pt}%
    \pgfmathsetlength{\ih}{\aspect\iw}%
}

\newcommand{\mkvideo}[4]{%
    \setupvideo{#1}{#2}{#3}%
    \hspace{-3ex}%
    {%
    \scriptsize%
    \begin{tikzpicture}%
        \strip{#1_u}{#2}{#3}{0}{Reconstr.}{}%
        \strip{#1_utem}{#2}{#3}{-\ih-3ex}{Temporal}{}%
        \strip{#1_uspa}{#2}{#3}{-2\ih-6ex}{Spatial}{}%
    \end{tikzpicture}%
    }%
}

\newcommand{\mkvideoorig}[4]{%
    \setupvideo{#1}{#2}{#3}
    \hspace{-3ex}%
    {%
    \scriptsize%
    \begin{tikzpicture}%
        \strip{#1_orig}{#2}{#3}{0}{Original}{#4}%
        \strip{#1_noisy}{#2}{#3}{-\ih-3ex}{Noisy}{}%
    \end{tikzpicture}%
    }%
}

\begin{figure}
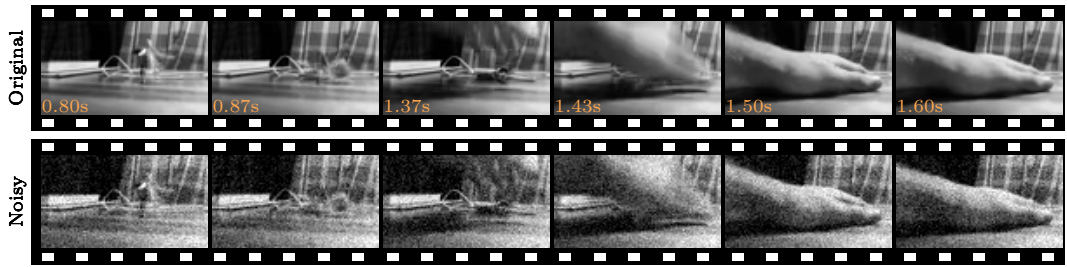
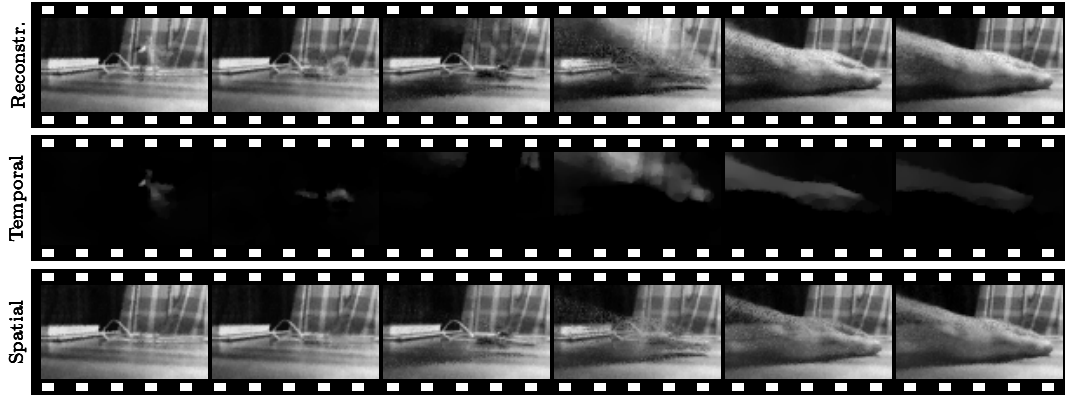
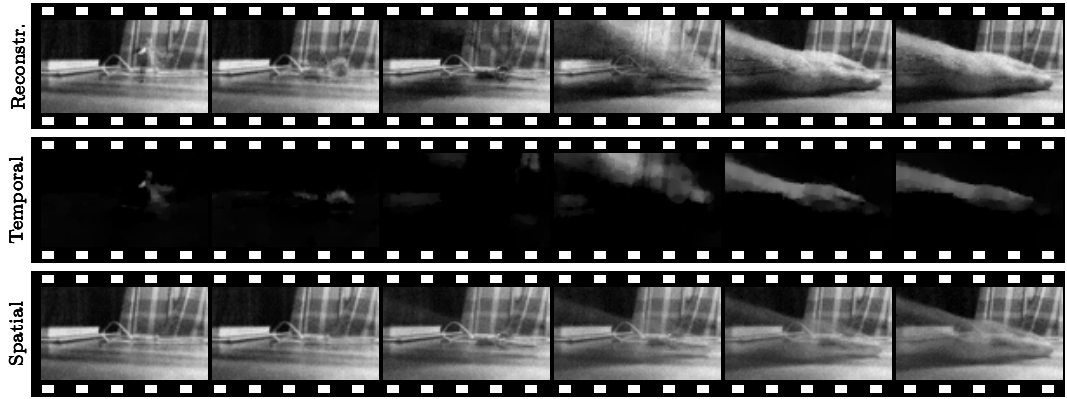
%
    \def\framespec{1/24,2/26,3/41,4/43,5/45,6/48}%
    \begin{subfigure}{\textwidth}%
        \mkvideoorig{resimg/res_fixed_Training_video_4_tiny}{6}{\framespec}{30}%
        \caption{Original and noisy videos}%
        \label{subfig:hand1}
    \end{subfigure}%
    \\
    \begin{subfigure}{\textwidth}%
        \mkvideo{resimg/res_fixed_Training_video_4_tiny_TVTV}{6}{\framespec}{30}%
        \caption{TVTV reconstruction}%
         \label{subfig:hand2}
    \end{subfigure}%
    \\%
    \begin{subfigure}{\textwidth}%
        \mkvideo{resimg/res_fixed_Training_video_4_tiny_L2TV}{6}{\framespec}{30}%
        \caption{$L^2$TV reconstruction}%
         \label{subfig:hand3}
    \end{subfigure}%
    \caption{``Hand'' test video reconstructions with optimally learned parameters using \eqref{eq:bilevopt-smooth} for the ICTVTV and IC$L^2$TV regularisation models. The reconstructions are depicted together with their temporal and spatial components. Note how $L^2$TV picks up the hand in the spatial component much earlier than TVTV.}
    \label{fig:hand}
\end{figure}

\begin{figure}
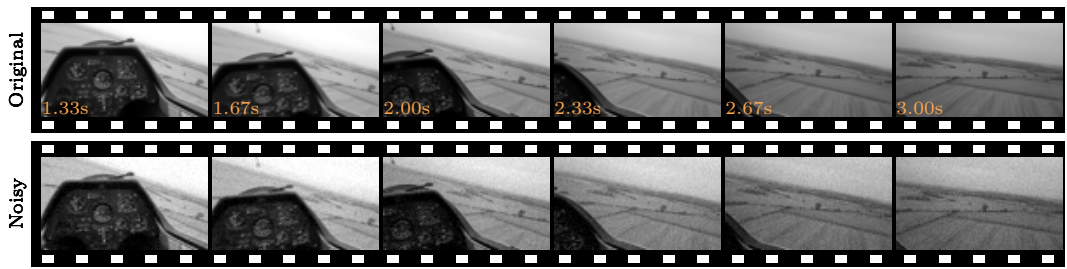
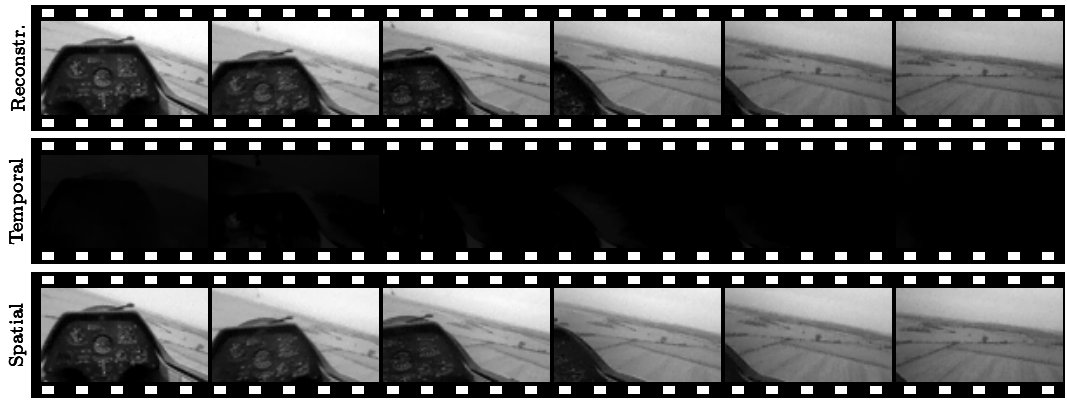
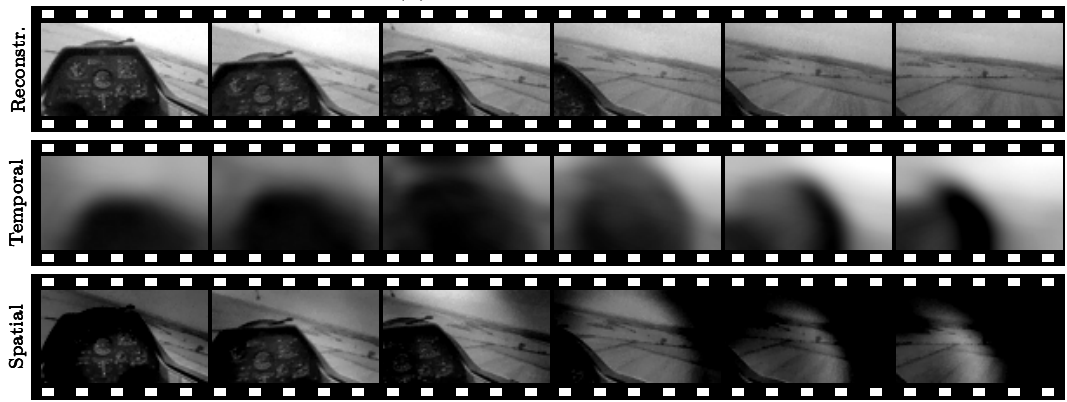
%
    \def\framespec{1/40,2/50,3/60,4/70,5/80,6/90}
    \begin{subfigure}{\textwidth}%
        \mkvideoorig{resimg/res_0.02_Flight}{6}{\framespec}{30}%
        \caption{Original and noisy videos}%
        \label{subfig:flight1}
    \end{subfigure}%
    \\
    \begin{subfigure}{\textwidth}
        \mkvideo{resimg/res_0.02_Flight_TVTV}{6}{\framespec}{30}%
        \caption{TVTV reconstruction}%
		\label{subfig:flight2}
    \end{subfigure}%
    \\
    \begin{subfigure}{\textwidth}%
        \mkvideo{resimg/res_0.02_Flight_L2TV}{6}{\framespec}{30}%
        \caption{$L^2$TV reconstruction}%
        \label{subfig:flight3}
    \end{subfigure}%
    \caption{``Flight'' test video reconstructions with optimally learned parameters using \eqref{eq:bilevopt-smooth} for the ICTVTV and IC$L^2$TV regularisation models. The reconstructions are depicted together with their temporal and spatial components. Note how TVTV manages to extract no temporal component, while $L^2$TV manages to extract the spatially stable rough features in the temporal component. Recall that we define the temporal component as $u$ if $\kappa \le 0.5$, and $u-w$ otherwise. Since $\kappa$ is approximately 0.5, see Table \ref{table:flight}, it can be argued that the nomenclature ``spatial'' and ``temporal'' components should be swapped for $L^2$TV, and the spatial component should be the blurry one, corresponding to the very approximate spatially constant information.}
    \label{fig:flight}
\end{figure}

\begin{figure}
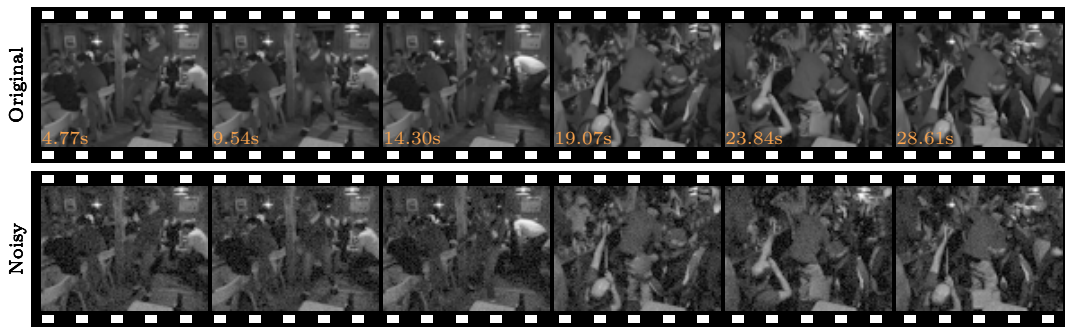
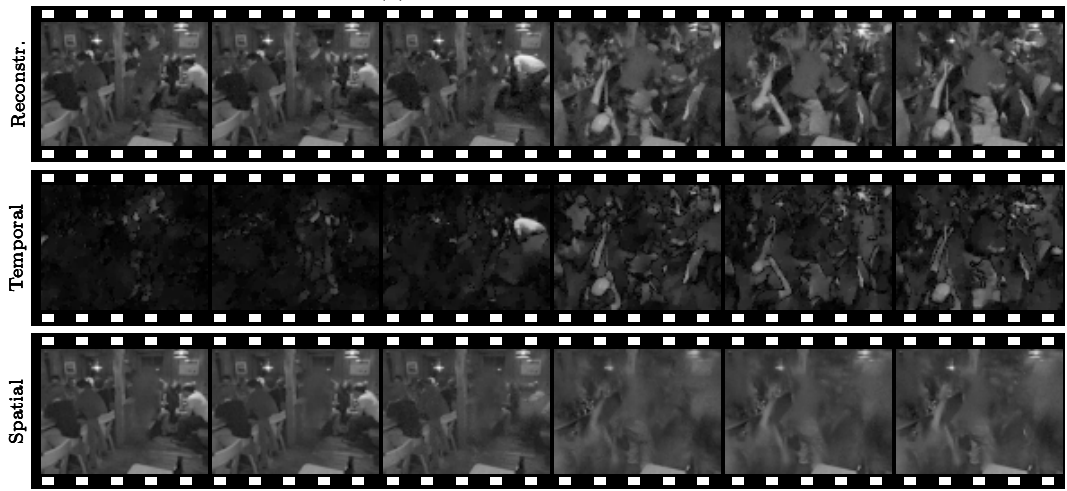
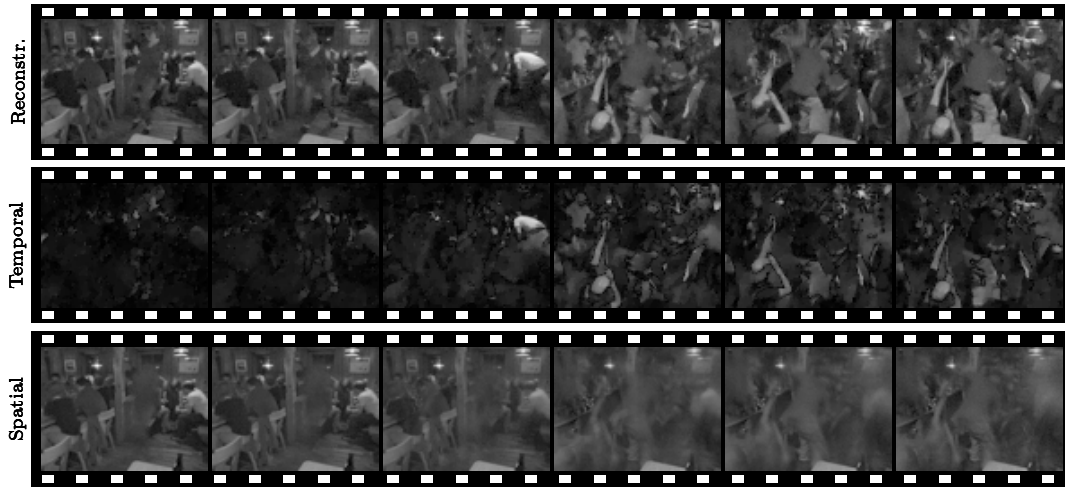
%
    \def\aspect{0.85}
        \def\framespec{1/12,2/24,3/36,4/48,5/60,6/72}
    \begin{subfigure}{\textwidth}%
        \mkvideoorig{resimg/res_hshake}{6}{\framespec}{2.517}%
        \caption{Original and noisy videos}%
    \end{subfigure}%
    \\
    \begin{subfigure}{\textwidth}
        \mkvideo{resimg/res_hshake_TVTV}{6}{\framespec}{2.517}%
        \caption{TVTV reconstruction}%
    \end{subfigure}%
    \\
    \begin{subfigure}{\textwidth}%
        \mkvideo{resimg/res_hshake_L2TV}{6}{\framespec}{2.517}%
        \caption{$L^2$TV reconstruction}%
    \end{subfigure}%
    \caption{``Harlem shake'' video reconstructions with optimally learned parameters using \eqref{eq:bilevopt-smooth} for the ICTVTV and IC$L^2$TV regularisation models. The reconstructions are depicted together with their temporal and spatial components. The displayed images have been gamma-corrected with factor $\gamma=0.6$ to improve legibility on paper.}
    \label{fig:harlemshake}
\end{figure}

For the next sequence, 'flight', the situation looks quite similar in terms of PSNR and SSIM, however, the components are very different in this case. Figure \ref{subfig:flight1} shows six consecutive frames of the original sequence, and the same frames from the sequence contaminated with noise. Figure \ref{subfig:flight2} shows the reconstructions with \eqref{eq:ictvtv} as a regulariser. Clearly, the temporal component in this case fails to pick up any information about the sequence, whereas all the information is stored in the spatial component. In case of \eqref{eq:icl2tv}, the optimal value for $\kappa$ is approximately 0.5 according to Table \ref{table:flight}. Hence, we can hardly speak of a temporal and a spatial component in this case, but rather of an $L^2$- and a TV-penalised component that are shown in Figure \ref{subfig:flight3}. The first component is rather blurry, and more or less approximating the homogeneous sky and a blurred version of the cockpit. The second component on the other hand approximates the sharp features and structures, but not the homogeneous background parts.

\begin{table}
    \centering
    \caption{``Flight'' test video optimal results. The star $^*$ in the optimal parameter means that the original $\kappa$ for the optimal result was outside the range $[0, 1]$, and the conversion \eqref{eq:negative-trick} has been used to derive the presented values.}
    \label{table:flight}
    \begin{tabular}{r|r|r|r|r}
        Model & $(\alpha_1, \alpha_2, \kappa)$ & Opt.~value & PSNR & SSIM \\
        \hline
        TVTV & (0.0141, 0.017, 0.337)$^*$ & 46.41 & 36.91 & 0.9569 \\
        L2TV & (4.81, 0.0163, 0.542) & 46.87 & 32.84 & 0.942 \\
    \end{tabular}
\end{table}

For the last sequence shown in Figure \ref{fig:harlemshake}, Harlem shake, we figure out that both \eqref{eq:ictvtv} and \eqref{eq:icl2tv} seem to perform almost equally well, indicating a small value of $\kappa$ which is confirmed by Table \ref{table:harlemshake}. Given the nature of the scene, the temporal part captures most of the scene, as there are only very few pixels that remain (almost) unchanged over time. Some of those are captured in the spatial component, like parts of the table in the background on the left hand side for instance.

\begin{table}
    \centering
    \caption{``Harlem shake'' test video optimal results. The star $^*$ in the optimal parameter means that the original $\kappa$ for the optimal result was outside the range $[0, 1]$, and the conversion \eqref{eq:negative-trick} has been used to derive the presented values.}
    \begin{tabular}{r|r|r|r|r}
        Model & $(\alpha_1, \alpha_2, \kappa)$ & Opt.~value & PSNR & SSIM \\
        \hline
        TVTV & (0.0649, 0.0122, 0.0319) & 39.52 & 37.67 & 0.9621 \\
        L2TV & (23, 0.0114, 0.0265) & 41.02 & 37.53 & 0.9604 \\
    \end{tabular}
    \label{table:harlemshake}
\end{table}

\section{Conclusions \& Outlook}
We have presented a bilevel optimisation strategy for optimising the regularisation parameters of infimal convolution-type regularisation methods for dynamic image regularisation. We have shown existence of solutions under additional regularity assumptions, and demonstrated the numerical performance of the proposed method for three distinctive video sequences. The sequences considered allow the assumption, that the use of the infimal convolution models for video denoising highly depends on the corresponding video sequence. As for sequences with stationary backgrounds, the decomposition into spatial and temporal components seems to work well, with the optimal $\kappa$ being close to 0 or 1 allowing for a clear distinction of a temporal and a spatial component. For sequences like the 'flight'-sequence that lack stationary parts, the distinction is not clear at all, which is also underpinned by the optimal $\kappa$ value being close to 0.5. 
In this setup the modelling assumption of multiple infimal convolutions of the same functional also breaks down, as they will be the same. A choice of $\kappa$ close to 0.5 only makes sense for infimal convolutions of functionals that promote very different information (like the $L^2$TV model in our case).
%

Future research should address different error measures for the comparison of the denoised video sequences to the ground truth, in order to see, if different error measures will lead to similar or different conclusions. Further can the proposed research be easily extended to general, bounded linear operators $K$, which has been omitted here for the sake of brevity. Research on infimal convolutions of different, complementary regularisation functionals seems to be another promising direction that future research can head for.

\section*{Acknowledgements}

M.~Benning, C.-B.~Sch\"{o}nlieb and T.~Valkonen acknowledge support from the {EPSRC} grant Nr.~EP/M00483X/1 and from the Leverhulme grant `Breaking the non-convexity barrier'. V.~Vla\v{c}i\'{c} was supported by a Bridgwater summer internship.

\section*{A data statement for the EPSRC}

{\color{red}The code and data will be put into a repository as the final version is submitted.}

\section*{References}

\bibliographystyle{iopart-num}

\bibliography{abbrevs,dynreg,bib,bib-own}


\end{document}